\documentclass[11pt,letterpaper]{amsart}
\usepackage{graphics}
\usepackage{amscd}
\usepackage{color}
\usepackage{comment}
\usepackage{oldgerm}
\usepackage{amsmath}
\usepackage{amssymb}
\usepackage{amsthm}
\newtheorem{thm}{Theorem}[section]
\newtheorem{prop}[thm]{Proposition}
\newtheorem{lem}[thm]{Lemma}

\newtheorem{definition}[thm]{Definition}
\newtheorem{rem}[thm]{Remark}

\newcommand{\quater}{\mathbb{H}}
\newcommand{\Z}{\mathbb{Z}}

\newcommand{\R}{\mathbb{R}}
\newcommand{\C}{\mathbb{C}}
\newcommand{\T}{\mathbb{T}}

\newcommand{\del}{\partial}
\newcommand{\delb}{\overline{\partial}}

\newcommand{\E}{\mathcal{E}}
\newcommand{\vol}{{\rm vol}}
\newcommand{\PP}{\mathbb{P}}
\title[Smooth maps minimizing the energy]{Smooth maps minimizing the energy and the calibrated geometry}
\author[K. Hattori]{Kota Hattori}
\address{Keio University, 
3-14-1 Hiyoshi, Kohoku, Yokohama 223-8522, Japan}
\email{hattori@math.keio.ac.jp}

\begin{document}
\maketitle

\begin{abstract}
We generalize the notion of calibrated submanifolds to smooth maps and show that the several examples of smooth maps 
appearing in the differential geometry become the examples 
of our situation. 
Moreover, we apply these notion to 
give the lower bound to the energy of 
smooth maps in the given homotopy class 
between Riemannian manifolds, 
and consider the energy functional which is minimized 
by the identity maps on the Riemannian manifolds with 
special holonomy groups. 
\end{abstract}

\tableofcontents

\section{Introduction}
In this article, we introduce the notion of calibrated geometry for smooth maps between Riemannian manifolds 
and consider the lower bound or the minimizers of several energy of smooth maps. 
Let $(X,g)$ and $(Y,h)$ be a compact Riemannian manifolds 
and $f\colon X\to Y$ be a smooth map. 
Then the $p$-energy of $f$ is defined by 
\begin{align*}
\E_p(f):=\int_X |df|^pd\mu_g
\end{align*}
for $p\ge 1$, 
where $\mu_g$ is the volume measure of $g$. 
A harmonic map is defined to be a critical point of $\E_2$, 
and the study of it is 
one of the significant region in differential geometry.  
In 1964, Eells and Sampson \cite{ES1964} 
have shown that there is a harmonic map 
$f'$ homotopic to $f$ if the sectional curvature of $h$ is nonpositive. 
Moreover, Hartman \cite{Hartman1967} showed 
that such harmonic maps minimize $\E_2|_{[f]}$, 
where $[f]$ is the homotopy class represented by $f$. 

In general, harmonic maps need not minimize the energy. 
For example, although the identity maps on any Riemannian manifolds are always 
harmonic, it is known that there is a smooth map 
$f_\varepsilon$ homotopic to the identity map 
such that $\E_2(f_\varepsilon)=\varepsilon$ on 
the $n$-sphere $S^n$ with the standard metric and $n\ge 3$. 
By the result shown by White \cite{White1986}, 
if $\pi_l(X)$ is trivial 
for all $1\le l\le k$, then $\inf \E_k|_{[1_X]}=0$, 
where $1_X$ is the identity map of $X$. 

Now, we consider how to give the lower bound to the energy 
restricting to a given homotopy class $[f]$ 
and the minimizer of them. 
Such a lower bound was first obtained by 
Lichnerowicz \cite{Lichnerowicz1969} in the case of 
$(X,g)$ and $(Y,h)$ are K\"ahler manifolds, 
then it was shown that any holomorphic maps between 
K\"ahler manifolds minimize $\E_2$ in their homotopy classes. 
Moreover, Croke \cite{Croke1987} showed that the identity 
map on the real projective space with the standard metric 
minimize $\E_2$ in its homotopy class, 
then Croke and Fathi \cite{CrokeFathi1990} 
introduced the new homotopy invariant 
called the intersection, which gives the lower bound to 
$\E_2|_{[f]}$ for a given homotopy class $[f]$. 
Recently, Hoisington \cite{hoisington2021} give the lower bound 
to $\E_p$ for an appropriate $p$ in the case of 
$X$ is real, complex, or quaternionic projective spaces 
with the standard metrics. 

In this article, we generalize the notion of calibrated geometry 
to smooth maps between smooth manifolds and give the lower bound to the several energy. 
The origin of calibrated geometry is the 
Wirtinger's inequality for the 
even dimensional 
subspaces in hermitian inner product spaces 
\cite{Wirtinger1936}, 
then it refined or generalized 
by many researchers. 
In \cite{harvey1982calibrated}, 
Harvey and Lawson defined 
calibrated submanifolds in the 
Calabi-Yau, $G_2$ or $Spin(7)$ manifolds 
which minimize the volume in their homology classes. 
Similarly, we define the new class of smooth maps, 
called calibrated maps, and show that they minimize 
the appropriate energy 
for the given situation. 
Moreover, we obtain the next results as applications. 

The first application is to obtain the lower bound to $p$-energy 
restricting to the given homotopy class. 
We assume $X$ is oriented. 
The pullback of $f$ induces a linear map 
$[f^*]^k\colon H^k(Y,\R)\to H^k(X,\R)$. 
By fixing basis of $H^k(X,\R)$, $H^k(Y,\R)$, 
we obtain the matrix $P([f^*]^k)$ of $[f^*]^k$ 
and put $|P([f^*]^k)|:=\sqrt{{\rm tr}({}^tP([f^*]^k)\cdot P([f^*]^k))}$. 
\begin{thm}
Let $(X,g)$ and $(Y,h)$ be as above. 
For any $1\le k\le \dim X$, there is a positive constant $C$ 
depending only on $k$, 
$(X,g)$, $(Y,h)$ and the basis of 
$H^k(X,\R)$, $H^k(Y,\R)$ 
such that for any $f\in C^\infty(X,Y)$ we have 
\begin{align*}
\E_k(f)\ge C|P([f^*]^k)|.
\end{align*}
In particular, if $[f^*]^k$ is nonzero, 
then $\inf (\E_k|_{[f]})$ is positive. 
\label{thm main1}
\end{thm}
In the above theorem, the compactness of $Y$ is not essential. 
See Theorem \ref{thm lower bdd of k-energy}. 

The second application is to show that 
the identity maps of some Riemannian manifolds with special holonomy groups minimize the appropriate energy. 
As we have already mentioned, the identity map 
on the real or complex projective space minimize $\E_2$ in its homotopy class 
by \cite{Croke1987} 
and \cite{Lichnerowicz1969}, respectively. 
It was shown by Wei \cite{Wei1998} that 
the identity map on the quaternionic projective space 
$\quater\PP^n$ with the standard metric is an unstable 
critical point of $\E_p$ 
for $1\le p< 2+4n/(n+1)$. 
Moreover, Hoisington gave the nontrivial lower bound of 
$\E_p|_{[1_{\quater\PP^n}]}$ for $p\ge 4$. 
Here, the quaternionic projective space 
is a typical example of quaternionic K\"ahler manifolds, 
which are Riemannian manifolds of dimension $4n$ 
whose holonomy group is contained in $Sp(n)\cdot Sp(1)$. 
Now, let $A$ be an $n\times m$ real-valued matrix 
and denote by $a_1,\ldots,a_m\in\R$ the 
nonnegative eigenvalues of ${}^tAA$, then put 
$|A|_p:=(\sum_{i=1}^m a_i^{p/2})^{1/p}$. 
Moreover, we define an energy $\E_{p,q}$ by 
\begin{align*}
\E_{p,q}(f):=\int_X|df|_p^qd\mu_g,
\end{align*}
then we have $\E_p=\E_{2,p}$. 
\begin{thm}\label{thm main2}
Let $(X,g)$ be a compact quaternionic K\"ahler manifold 
of dimension $4n\ge 8$. 
Then the identity map of $X$ minimizes $\E_{4,4}$ 
in its homotopy class. 
\end{thm}
We can also show the similar theorem in the case of 
other holonomy groups. 
If $(X,g)$ is a compact $G_2$ manifold, then $1_X$ 
minimizes $\E_{3,3}|_{[1_X]}$ and 
if $(X,g)$ is a compact $Spin(7)$ manifold, then $1_X$ 
minimizes $\E_{4,4}|_{[1_X]}$ 
(See Theorem \ref{thm id calib}). 
Moreover, it is easy to see that if the identity map minimizes 
$\E_{p,q}$, then it also minimizes $\E_{p',q'}$ 
for all $p'\ge p$ and $q'\ge q$ by the 
H\"older's inequality. 
Of course, we can also consider the case of 
K\"ahler, Calabi-Yau, hyper-K\"ahler manifolds respectively, however, the results in these cases also follow from \cite{Lichnerowicz1969}. 

This paper is organized as follows. 
In Section \ref{sec energy of map}, 
we define the notion of calibrated maps, which is the analogy 
of the calibrated submanifolds. 
In Section \ref{sec ex}, we explain some examples of 
calibrated maps. 
In particular, we show that holomorphic maps 
between K\"ahler manifolds and the inclusion maps of 
calibrated submanifolds can be regarded as calibrated maps. 
Moreover, we can also show that the fibration 
whose regular fibers are calibrated submanifolds are calibrated maps. 
We prove Theorem \ref{thm main1} in Section \ref{sec lower bdd}, 
and Theorem \ref{thm main2} in Section \ref{sec id map}. 
In Section \ref{sec intersection}, 
we compare the homotopy invariant introduced in 
\cite{CrokeFathi1990} with the invariants defined in this paper.

\vspace{0.2cm}
\paragraph{\bf Acknowledgment}
I would like to thank Professor Frank Morgan for his advice on this paper. 
This work was supported by JSPS KAKENHI Grant Numbers JP19K03474, JP20H01799.

\section{Calibrated maps}\label{sec energy of map}

Let $X,Y$ be smooth manifold 
of $\dim X=m$ and $\dim Y=n$. 
Throughout of this paper, 
we suppose $X$ is compact and oriented. 
We fix a volume form ${\rm vol}\in\Omega^m(X)$ on $X$, namely, 
a nowhere vanishing $m$-form 
which determines an orientation and a measure of $X$. 
For $m$-forms $v_1,v_2\in\Omega^m(X)$, there are 
$\varphi_i\subset C^\infty(X)$ with $v_i=\varphi_i\vol$. 
Then we write $v_1\le v_2$ if 
$\varphi_1(x)\le \varphi_2(x)$ for all $x\in X$. 

If a map 
$\sigma\colon C^\infty(X,Y)\to L^1(X)$ is given, 
then we can define an energy 
$\E\colon C^\infty(X,Y)\to \R$ by 
\begin{align*}
\E(f):=\int_X \sigma(f)\vol. 
\end{align*}

Now, $f_0,f_1\in C^\infty(X,Y)$ are said to be {\it homotopic} 
if there is a smooth map $F\colon [0,1]\times X\to Y$ 
such that $F(0,\cdot)=f_0$ and $F(1,\cdot)=f_1$. 
By Whitney approximation theorem, it is equivalent to 
the existence of the continuous homotopy 
joining $f_0$ and $f_1$. 
For $f\in C^\infty(X,Y)$, denote by $[f]\subset C^\infty(X,Y)$ 
the homotopy equivalent class represented by $f$. 
In this paper we consider the lower bound to $\E|_{[f]}$ or the minimum 
of $\E|_{[f]}$. 

Denote by $1_X\colon X\to X$ the identity map on $X$. 
We define a smooth map $(1_X,f)\colon X\to X\times Y$ 
by 
\begin{align*}
(1_X,f)(x):=(x,f(x)).
\end{align*}

The next definition is the analogy of 
\cite{harvey1982calibrated}. 
\begin{definition}
\normalfont
$\Phi
\in\Omega^m(X\times Y)$ is a {\it $\sigma$-calibration} 
if $d\Phi=0$ and 
\begin{align*}
(1_X,f)^*\Phi\le \sigma(f)\vol
\end{align*}
for any smooth map $f\colon X\to Y$. 
Moreover, $f$ is a {\it $(\sigma,\Phi)$-calibrated map} 
if 
\begin{align*}
(1_X,f)^*\Phi = \sigma(f)\vol.
\end{align*}
\end{definition}

\begin{thm}
Let $\sigma$ be an energy density and 
$\Phi$ be a $\sigma$-calibration. 
\begin{itemize}
\setlength{\parskip}{0cm}
\setlength{\itemsep}{0cm}
 \item[$({\rm i})$] The constant 
$\int_X(1_X,f)^*\Phi$ is determined by the homotopy class $[f]$. 
In other words, 
$\int_X(1_X,f_0)^*\Phi=\int_X(1_X,f_1)^*\Phi$ if $[f_0]=[f_1]$. 
 \item[$({\rm ii})$] We have $\inf\E|_{[f]} \ge \int_X(1_X,f)^*\Phi$ for any 
$f\in C^\infty(X,Y)$. 
 \item[$({\rm iii})$] We have $\E(f)=\int_X(1_X,f)^*\Phi$ iff $f$ is $(\sigma,\Phi)$-calibrated map. In particular, any 
$(\sigma,\Phi)$-calibrated map minimizes $\E$ in its homotopy class. 
\end{itemize}
\end{thm}
\begin{proof}
$({\rm i})$ If $f_0,f_1$ are homotopic, then $(1_X,f_0)$ and $(1_X,f_1)$ are homotopic, 
accordingly $(1_X,f_0)^*\Phi $ and $(1_X,f_1)^*\Phi$ represent 
the same cohomology class by \cite[Corollary 4.1.2]{BT1982}. 

$({\rm ii})$ is shown by the definition of 
$\sigma$-calibration. 

$({\rm iii})$ 
By the point-wise inequality 
$(1_X,f)^*\Phi\le \sigma(f)\vol$, 
we have $\E(f)=\int_X(1_X,f)^*\Phi$ iff $(1_X,f)^*\Phi= \sigma(f)\vol$. 
\end{proof}

\section{Examples}\label{sec ex}

One of the typical example of the energy of maps is 
$p$-energy defined for the smooth maps between 
Riemannian manifolds. 
Let $(X,g)$ and $(Y,h)$ be Riemannian manifolds 
and $f\colon X\to Y$ be a smooth map. 
Then the pullback $f^*h$ is a section of 
$T^*X\otimes T^*X$, we can take the trace ${\rm tr}_g(f^*h)$. 
For $p\ge 1$, put $\sigma_p(f):=\{ {\rm tr}_g(f^*h)\}^{p/2}$. 
We assume that $X$ is oriented and denote by $\vol_g$ 
the volume form of $g$. 
The $p$-energy $\E_p(f)$ is defined by 
\begin{align*}
\E_p(f):=\int_X\sigma_p(f)\vol_g.
\end{align*}

Now, the differential $df_x$ is an element of $T^*_xX\otimes T_{f(x)}Y$ for every $x\in X$. 
Since $g_x$ and $h_{f(x)}$ 
induces the natural inner product and the norm on 
$T^*_xX\otimes T_{f(x)}Y$, 
Then we may also write 
$\sigma_p(f)(x)=|df_x|^p$.

By the H\"older's inequality, we have 
\begin{align*}
\E_p(f)&\le \vol_g(X)^{1-p/q}\E_q(f)^{p/q}
\end{align*}
for $1\le p\le q$. Thus we have the following proposition. 
\begin{prop}
Let $\Phi\in\Omega^m(X\times Y)$ be 
a $\sigma_p$-calibration. 
Then 
\begin{align*}
\vol_g(X)^{-1+p/q}\int_X(1_X,f)^*\Phi\le \E_q(f)^{p/q}
\end{align*}
for any $q\ge p$ and $f\in C^\infty(X,Y)$. 
\label{prop holder}
\end{prop}

\subsection{holomorphic maps}
Here, assume that $X,Y$ are complex manifolds 
and $g,h$ are K\"ahler metrics. 
Let $m=\dim_\C X$ and $n=\dim_\C Y$. 
Then we have the decomposition 
\begin{align*}
T^*X\otimes \C&=\Lambda^{1,0}T^*X\oplus\Lambda^{0,1}T^*X,\\
TY\otimes \C&=T^{1,0}Y\oplus T^{0,1}Y,
\end{align*}
accordingly the derivative $df\in \Gamma(T^*X\otimes f^*TY)$ is decomposed into 
\begin{align*}
df&=(\del f)^{1,0}+(\del f)^{0,1}+(\delb f)^{1,0}+(\delb f)^{0,1}\\
&\in (\Lambda^{1,0}T^*X\otimes T^{1,0}Y)
\oplus (\Lambda^{1,0}T^*X\otimes T^{0,1}Y)\\
&\quad\quad\oplus (\Lambda^{0,1}T^*X\otimes T^{1,0}Y)
\oplus (\Lambda^{0,1}T^*X\otimes T^{0,1}Y).
\end{align*}
Since $df$ is real, we have 
\begin{align*}
\overline{(\del f)^{1,0}}=(\delb f)^{0,1},
\quad\overline{(\del f)^{0,1}}=(\delb f)^{1,0}.
\end{align*}
Denote by $\omega_g,\omega_h$ the 
K\"ahler form of $g,h$, respectively, 
then the volume form is given by 
$\vol_g=\frac{1}{m!}\omega_g^m$. 
The following observation was given by Lichnerowicz.
\begin{thm}[{\cite{Lichnerowicz1969}}]
For any smooth map $f\colon X\to Y$, we have 
\begin{align*}
\omega_g^{m-1}\wedge f^*\omega_h
&=(m-1)!(|(\del f)^{1,0}|^2-|(\delb f)^{1,0}|^2)\vol_g,\\
|df|^2&=2|(\del f)^{1,0}|^2+2|(\delb f)^{1,0}|^2.
\end{align*}
In particular, we have 
\begin{align*}
\E_2(f)\ge \frac{2}{(m-1)!}\int_X\omega_g^{m-1}\wedge f^*\omega_h
\end{align*}
and the equality holds 
iff $f$ is holomorphic. 
\label{thm Lich}
\end{thm}
Now, we consider $\omega_g^{m-1}\wedge \omega_h\in \Omega^m(X\times Y)$. 
The first two equalities in Theorem \ref{thm Lich} 
implies that $\frac{2}{(m-1)!}\omega_g^{m-1}\wedge \omega_h$ 
is a $\sigma_2$-calibration. 
Moreover, the second statement implies that 
$f$ is a $(\sigma_2,\frac{2}{(m-1)!}\omega_g^{m-1}\wedge \omega_h)$-calibrated map iff $f$ is holomorphic. 
One can also see that 
$f$ is $(\sigma_2,-\frac{2}{(m-1)!}\omega_g^{m-1}\wedge \omega_h)$-calibrated map iff $f$ is anti-holomorphic. 

\subsection{Calibrated submanifolds}
In this subsection, we see the relation between 
the calibrated submanifolds in the sense of  \cite{harvey1982calibrated} and 
the calibrated maps. 
We assume $(Y^n,h)$ is a Riemannian manifold. 
\begin{definition}[\cite{harvey1982calibrated}]
\normalfont
For an integer $0<m<n$, 
$\psi
\in\Omega^m(Y)$ is a {\it calibration} 
if $d\psi=0$ and 
\begin{align*}
\psi|_V\le \vol_{h|_V}
\end{align*}
for any $y\in Y$ and 
$m$-dimensional oriented 
subspace $V\subset T_yY$. 
Here, $h|_V$ is the induced metric on $V$ and 
$\vol_{h|_V}$ is its volume form whose orientation 
is compatible with the one equipped with $V$. 
Moreover, an oriented submanifold 
$X\subset Y$ is a {\it calibrated submanifold} 
if 
\begin{align*}
\psi|_{T_xX}=\vol|_{h|_{T_xX}}
\end{align*}
for any $x\in X$. 
\label{def HL cal}
\end{definition}

Now, if $X$ is an oriented manifold 
with a volume form $\vol\in\Omega^m(X)$, 
then for every linear map 
$A\colon T_xX\to T_yY$ 
can be regarded as an $n\times m$-matrix 
by taking a 
basis $e_1,\ldots,e_m$ of $T_xX$ and an orthonormal basis of $T_yY$ with $\vol_x(e_1,\ldots,e_m)=1$. 
Then $\sqrt{\det( {}^tA\cdot A)}$ does not 
depend on the choice of these basis. 
Therefore, for $f\in C^\infty(X,Y)$, we can define 
the energy density 
$\tau_m(f)(x):=\sqrt{\det({}^tdf_x\cdot df_x)}$ and the energy 
$\E_{\tau_m}(f):=\int_X\tau_m(f)\vol$.

\begin{prop}
Let $(X,\vol)$ be an oriented manifold equipped with a volume form and $\psi\in\Omega^m(Y)$ be closed. 
Assume that $\dim_\R X=m<n$ and 
denote by $\pi_Y\colon X\times Y\to Y$ the 
natural projection. 
Then $\psi$ is a calibration iff 
$\pi_Y^*\psi\in\Omega^m(X\times Y)$ 
is a $\tau_m$-calibration. 
Moreover, for any embedding $f\colon X\to Y$, 
the following conditions are equivalent. 
\begin{itemize}
\setlength{\parskip}{0cm}
\setlength{\itemsep}{0cm}
 \item[$({\rm i})$] 
 $f(X)$ is a calibrated submanifold, where the orientation of $f(X)$ is determined such that $f$ preserves the orientation. 
 \item[$({\rm ii})$] $f$ is 
a $(\tau_m,\pi_Y^*\psi)$-calibrated map. 
\end{itemize}
\end{prop}
\begin{proof}
Note that $(1_X,f)^*(\pi_Y^*\psi)=f^*\psi$ 
and 
$\tau_m(f)\vol_g=\vol_{f^*h}$. 
Hence $\psi$ is a calibration iff 
$\pi_Y^*\psi\in\Omega^m(X\times Y)$ 
is a $\tau_m$-calibration. 
Moreover, 
suppose that $f$ is an embedding. 
Then $f$ is a $(\tau_m,\pi_Y^*\psi)$-calibrated map iff $f(X)$ is a calibrated submanifold. 
\end{proof}

\subsection{Fibrations}
Let $(X^m,g)$ be an oriented Riemannian manifold and 
$Y^n$ be a smooth manifold equipped with a volume form 
$\vol_Y\in\Omega^n(Y)$. 
Here, we suppose $n< m$ and 
let 
$\varphi\in\Omega^{m-n}$ be a calibration 
in the sense of Definition \ref{def HL cal}. 
Fix an orthonormal basis of $T_xX$ 
and a basis $e_1',\ldots,e_n'\in T_yY$ 
with $\vol_Y(e_1',\ldots,e_n')=1$, we can regard 
a linear map $A\colon T_xX\to T_yY$ as 
an $m\times n$-matrix. 
Then the value of $\sqrt{\det(A\cdot{}^tA)}$ 
does not depend on the choice of above basis. 
For a smooth map $f\colon X\to Y$, put 
$\tilde{\tau}_{m,n}(f)|_x:=\sqrt{\det(df_x\cdot{}^tdf_x)}$ 
and $\Phi:=\vol_Y\wedge \varphi$. 

Put
\begin{align*}
X_{\rm reg}:=\{ x\in X|\, x\mbox{ is a regular point of }f\}.
\end{align*}
Note that $X_{\rm reg}$ is open in $X$. 
If $x\in X_{\rm reg}$, 
we have the orthogonal decomposition $T_xX={\rm Ker}(df_x)\oplus H$ and $df_x|_H\colon H\to T_{f(x)}Y$ is a linear 
isomorphism. 
Put $y=f(x)$ and suppose that $f^{-1}(y)$ is a calibrated submanifold 
with respect to the suitable orientation. 
We say that $df_x$ is {\it orientation preserving} 
if there is a basis 
$v_1,\ldots,v_m$ of $T_xX$ such that 
\begin{align*}
v_1,\ldots,v_n&\in H,
\quad \vol_Y(df_x(v_1),\ldots,df_x(v_n))>0,\\
v_{n+1},\ldots,v_m&\in {\rm Ker}(df_x),\quad
\varphi_x(v_{n+1},\ldots,v_m)>0,\\
\vol_g(v_1,\ldots,v_m)&>0.
\end{align*}
\begin{prop}
$\Phi$ is a $\tilde{\tau}_{m,n}$-calibration. 
Moreover, a smooth map $f\colon X\to Y$ 
is a $(\tilde{\tau}_{m,n},\Phi)$-calibrated map 
iff 
\begin{itemize}
\setlength{\parskip}{0cm}
\setlength{\itemsep}{0cm}
 \item[$({\rm i})$] 
$f^{-1}(y)\cap X_{\rm reg}$ is a calibrated submanifold with respect to $\varphi$ and the suitable orientation for any $y\in Y$, 
 \item[$({\rm ii})$] 
$df_x$ is orientation preserving for any $x\in X_{\rm reg}$.
\end{itemize}
\end{prop}

\begin{proof}
If $x\in X$ is a critical point of $f$, then we can see 
\begin{align*}
f^*\vol_Y\wedge \varphi|_x=\tilde{\tau}_{m,n}(f)\vol_g|_x=0.
\end{align*}
Fix a regular point $x$ and an oriented 
orthonormal basis 
$e_1,\ldots,e_m\in T_xX$ such that 
$e_{m-n+1},\ldots,e_m\in {\rm Ker}(df_x)$. 
Then we have 
\begin{align*}
f^*\vol_Y\wedge \varphi(e_1,\ldots,e_m)
&=\vol_Y(df_x(e_1),\ldots,df_x(e_n))
\varphi(e_{n+1},\ldots,e_m),\\
\tilde{\tau}_{m,n}(f)|_x&=
\left| \vol_Y(df_x(e_1),\ldots,df_x(e_n))\right|.
\end{align*}
Since $\varphi$ is a calibration, we have 
$\varphi(\pm e_{n+1},e_{n+2},\ldots,e_m)\le 1$, 
hence $|\varphi(e_{n+1},\ldots,e_m)|\le 1$. 
Therefore, 
\begin{align*}
f^*\vol_Y\wedge \varphi(e_1,\ldots,e_m)
&\le \left| \vol_Y(df_x(e_1),\ldots,df_x(e_{m-n}))\right|
= \tilde{\tau}_{m,n}(f)|_x,
\end{align*}
which implies that 
$\Phi$ is a $\tilde{\tau}_{m,n}$-calibration.

Next we consider the condition 
\begin{align*}
f^*\vol_Y\wedge \varphi|_x=\tilde{\tau}_{m,n}(f)\vol_g|_x,
\end{align*}
where $x$ is a regular value of $f$. 
In this case we have the 
orthogonal decomposition 
$T_xX={\rm Ker}(df_x)\oplus H$, 
where $H$ is an $n$-dimensional subspace. 
We can take an orthonormal basis 
$e_1,\ldots,e_m\in T_xX$ such that 
\begin{align*}
e_1,\ldots,e_n&\in H,\\
e_{n+1},\ldots,e_m&\in {\rm Ker}(df_x),\\
a:=\vol_Y(df_x(e_1),\ldots,df_x(e_{m-n}))&>0,\\
\vol_g(e_1,\ldots,e_m)&>0.
\end{align*}
Then we have 
\begin{align*}
f^*\vol_Y\wedge \varphi(e_1,\ldots,e_m)
&=a
\varphi(e_{n+1},\ldots,e_m),\\
\tilde{\tau}_{m,n}(f)\vol_g(e_1,\ldots,e_m)&=|a|=a.
\end{align*}
Therefore, we have 
\begin{align*}
f^*\vol_Y\wedge \varphi|_x
&=\tilde{\tau}_{m,n}(f)\vol_g|_x
\end{align*}
iff $\varphi(e_{n+1},\ldots,e_m)=1$. 
Now we have taken $x\in X_{\rm reg}$ arbitrarily, 
hence we have 
\begin{align*}
f^*\vol_Y\wedge \varphi
&=\tilde{\tau}_{m,n}(f)\vol_g
\end{align*}
iff $f^{-1}(y)\cap X_{\rm reg}$ is a calibrated submanifold 
for any $y\in Y$ 
and $df_x$ is orientation preserving for all $x\in X_{\rm reg}$. 
\end{proof}

\subsection{Totally geodesic maps between tori}\label{subsec tori}
Let $\T^n=\R^n/\Z^n$ be the $n$-dimensional torus 
and we consider smooth maps from $\T^m$ to $\T^n$. 
Let $G=(g_{ij})\in M_m(\R)$ and $H=(h_{ij})\in M_n(\R)$ 
be positive symmetric matrices. 
Denote by $x=(x^1,\ldots,x^m)$ and 
$y=(y^1,\ldots,y^n)$ the Cartesian coordinate 
on $\R^m$ and $\R^n$, respectively, 
then we have closed $1$-forms $dx^i\in\Omega^1(\T^m)$ 
and $dy^i\in\Omega^1(\T^n)$. 
We define the flat Riemannian metrics 
$g=\sum_{i,j}g_{ij}dx^i\otimes dx^j$ on $\T^m$ 
and $h=\sum_{i,j}h_{ij}dy^i\otimes dy^j$ on $\T^n$. 

For a smooth map $f\colon \T^m\to \T^n$, 
we have the pullback $f^*\colon H^1(\T^n,\R)\to H^1(\T^m,\R)$. 
Here, since 
\begin{align*}
H^1(\T^m,\Z)&={\rm span}_\Z\{ [dx^1],\ldots,[dx^m]\},\\
H^1(\T^n,\Z)&={\rm span}_\Z\{ [dy^1],\ldots,[dy^n]\},
\end{align*}
there is $P=(P_i^j)\in M_{m,n}(\Z)$ such that 
$f^*[dy^j]=\sum_i P_i^j[dx^i]$. 
The matrix $P$ is determined by the homotopy class of $f$. 
Now, let $*_g$ be the Hodge star operator of $g$ 
and put 
\begin{align}
\Phi:=\sum_{i,j,k}h_{jk} P_i^j *_gdx^i\wedge dy^k
\in \Omega^m(\T^m\times \T^n).\label{eq phi on torus}
\end{align}
Then we can check that 
\begin{align*}
\int_{\T^m}(1_{\T^m},f)^*\Phi
&=\sum_{i,j,k,l}h_{jk} P_i^jP_l^k\int_{\T^m} *_gdx^i\wedge dx^l\\
&=\sum_{i,j,k,l}h_{jk} P_i^jP_l^k g^{il}\vol_g(\T^m)\\
&={\rm tr}({}^tPG^{-1}PH)\vol_g(\T^m)=:\| P\|^2\vol_g(\T^m)\ge 0.
\end{align*}
Consequently, by the positivity of $G^{-1}$ and $H$, 
$\int_{\T^m}(1_{\T^m},f)^*\Phi=0$ iff $P=0$. 
\begin{prop}
Assume that $f^*\colon H^1(\T^n,\R)\to H^1(\T^m,\R)$ 
is not the zero map. 
$\| P\|^{-1}\Phi$ is a $\sigma_1$-calibration and 
$f$ is a $(\sigma_1,\| P\|^{-1}\Phi)$-calibrated map 
if $f(x)=Px+a$ for some $a\in \T^m$. 
Moreover, $f$ minimizes $\E_2$ in its homotopy class 
iff $f(x)=Px+a$ for some $a\in \T^m$. 
\label{prop lower torus}
\end{prop}
\begin{proof}
We fix $x\in\T^m$ and put $df_x:=A=(A_i^j)\in M_{n,m}(\R)$, 
and show $(1_{\T^m},f)^*\Phi\le \sigma_1(f)\vol_g$ at $x$. 
Since 
\begin{align*}
(1_{\T^m},f)^*\Phi|_x
&=\sum_{i,j,k,l}h_{jk} P_i^jA_l^k *_gdx^i\wedge dx^l|_x\\
&=\left(\sum_{i,j,k,l}h_{jk} P_i^jA_l^k g^{il}\right)\vol_g|_x\\
&=\left({\rm tr}({}^tPG^{-1}AH)\right)\vol_g|_x.
\end{align*}
Here, by the Cauchy-Schwarz inequality, 
we have 
\begin{align*}
{\rm tr}({}^tPG^{-1}AH)\le \sqrt{\| P\|\| A\|},
\end{align*}
and the equality holds iff $A=\lambda P$ for a 
constant $\lambda\ge 0$. 
Therefore, we have 
\begin{align*}
(1_{\T^m},f)^*\Phi
&\le \| P\|\sigma_1(f)\vol_g,
\end{align*}
which implies that $\| P\|^{-1}\Phi$ is 
a $\sigma_1$-calibration. 
Moreover, the equality holds iff 
$df_x=\lambda_x \cdot {}^t P$ for some $\lambda_x\ge 0$. 
Therefore, $f(x)={}^tPx+a$ for some $a\in\T^m$ is 
a $(\sigma_1,\| P\|^{-1}\Phi)$-calibrated map. 

For any $f\in C^\infty(\T^m,\T^n)$, we have
\begin{align*}
\int_X(1_{\T^m},f)^*\Phi
&\le \| P\|\int_X\sigma_1(f)\vol_g
\le \| P\|\sqrt{\vol_g(\T^m)\E_2(f)}
\end{align*}
by the Cauchy-Schwartz inequality. 
Moreover, we have the following equality 
\begin{align*}
\int_X(1_{\T^m},f)^*\Phi=\| P\|\sqrt{\vol_g(\T^m)\E_2(f)}
\end{align*}
iff 
$df_x=\lambda_x\cdot {}^tP$ for some $\lambda_x\ge 0$ and 
$\sigma_1(f)$ is a constant function on $\T^m$. 
Since $\sigma_1(f)(x)=\lambda_x\| P\|$, 
if $\sigma_1(f)$ is constant, then 
$\lambda_x=\lambda$ is independent of $x$. 
Hence we may write $f(x)=\lambda \cdot {}^tPx+a$ for some 
$a\in \T^m$. 
Moreover, since $f^*=P$ on $H^1(\T^n)$, we have $\lambda=1$. 
\end{proof}

In the above proposition, we cannot show that 
every $(\sigma_1,\| P\|^{-1}\Phi)$-calibrated map 
is given by $f(x)={}^tPx+a$ for some $a$. 
The following case is a counterexample. 

Suppose $m=n=1$ and let $P=1$. 
If we put $f(x)=x+\frac{1}{2\pi}\sin (2\pi x)$, then it gives a 
smooth map $\T^1\to \T^1$ homotopic to 
the identity map. 
Then one can easy to check that 
$f$ is a $(\sigma_1,\| P\|^{-1}\Phi)$-calibrated map 
since $f'(x)\ge 0$.

\section{The lower bound of $p$-energy}\label{sec lower bdd}
In this section, we give the lower bound of $p$-energy 
in the general situation. 
Let $(X,g)$ and $(Y,h)$ be Riemannian manifolds 
and assume $X$ is compact and oriented. 
Now we have the decomposition 
\begin{align*}
\Lambda^kT^*_{(x,y)}(X\times Y)
\cong \bigoplus_{l=0}^k\Lambda^lT^*_xX\otimes \Lambda^{k-l}T^*_yY,
\end{align*}
then denote by $\Omega^{l,k-l}(X\times Y)\subset \Omega^k(X\times Y)$ the set consisting of 
smooth sections of $\Lambda^lT^*_xX\otimes \Lambda^{k-l}T^*_yY$. 
For $\Phi\in \Omega^k(X\times Y)$, let 
$|\Phi_{(x,y)}|$ be the norm with respect to the metric 
$g\oplus h$ on $X\times Y$. 
\begin{lem}
Let $\Phi\in \Omega^{m-k,k}(X\times Y)$ be closed and 
$\sup_{x,y}|\Phi_{(x,y)}|<\infty$. 
Then there is a constant $C>0$ depending only on 
$\Phi,m,n,k$ such that $C\Phi$ is a $\sigma_k$-calibration. 
\label{lem cal k}
\end{lem}
\begin{proof}
Fix $x\in X$ and let 
$\{ e_1,\ldots,e_m\}$ and $\{ e_1',\ldots,e_n'\}$ 
be an orthonormal basis of 
$T_xX$ and $T_{f(x)}Y$, respectively. 
Put 
\begin{align*}
\mathcal{I}_k^m
:=\left\{ I=(i_1,\ldots,i_k)\in\Z^k\left|\, 
0\le i_1<\cdots<i_k\le m\right.\right\}.
\end{align*}
For $I=(i_1,\ldots,i_k)\in \mathcal{I}_k^m$, 
$J=(j_1,\ldots,j_k)\in\mathcal{I}_k^n$, we write 
\begin{align*}
e_I:=e_{i_1}\wedge\cdots\wedge e_{i_k},\quad
e_J':=e_{j_1}'\wedge\cdots\wedge e_{j_k}'.
\end{align*}
Then we have 
\begin{align*}
\Phi_{(x,f(x))}=\sum_{I\in\mathcal{I}_k^m,J\in\mathcal{I}_k^n}\Phi_{IJ}(*_ge_I)\wedge e'_J
\end{align*}
for some $\Phi_{IJ}\in\R$ 
and 
\begin{align*}
\left\{ (1_X,f)^*\Phi\right\}_x
&=\sum_{I,J}\Phi_{IJ}(*_ge_I)\wedge df_x^*e'_J.
\end{align*}
If we denote by 
$(df_x)_{IJ}$ the $k\times k$ matrix 
whose $(p,q)$-component is given by 
$g(df_x(e_{i_q}), e'_{j_p})$, then we have 
\begin{align*}
(*_ge_I)\wedge df_x^*e'_J
&=\det((df_x)_{IJ})\vol_g|_x\le k!|df_x|^k\vol_g|_x,
\end{align*}
therefore, 
we can see 
\begin{align*}
\left\{ (1_X,f)^*\Phi\right\}_x
&\le \left(\sum_{I,J}|\Phi_{IJ}|\right) k! |df_x|^k\vol_g|_x.
\end{align*}
Since $|\Phi_{x,f(x)}|^2=\sum_{I,J}|\Phi_{I,J}|^2$, 
we have 
\begin{align*}
(1_X,f)^*\Phi
&\le k! (\#\mathcal{I}_k^m)(\#\mathcal{I}_k^n)\sup_{x,y}|\Phi_{(x,y)}| \sigma_k(f)\vol_g,
\end{align*}
which implies the assertion. 
\end{proof}

For $f\in C^\infty(X,Y)$, denote by $[f^*]^k$ 
the pullback $H^k(Y,\R)\to H^k(X,\R)$ of $f$. 
For a closed form $\alpha\in\Omega^k(Y)$, 
denote by $[\alpha]\in H^k(Y,\R)$ its cohomology class. 
Put 
\begin{align*}
H^k_{\rm bdd}(Y,\R)
:=\left\{ [\alpha]\in H^k(Y,\R)|\, \alpha\in \Omega^k(Y),\, d\alpha=0,\, \sup_{y\in Y}h(\alpha_y,\alpha_y)<\infty\right\}. 
\end{align*}
This is a subspace of $H^k(Y,\R)$, and 
we have $H^k_{\rm bdd}(Y,\R)=H^k(Y,\R)$ if $Y$ is compact. 
Denote by $[f^*]^k_{\rm bdd}$ the restriction of 
$[f^*]^k$ to $H^k_{\rm bdd}(Y,\R)$. 
Fixing a basis of $H^k(X,\R)$ and $H^k_{\rm bdd}(Y,\R)$, 
we obtain the matrix $P=P([f^*]^k_{\rm bdd})\in M_{N,d}(\R)$ of 
$[f^*]^k_{\rm bdd}$, 
where $d=\dim H^k_{\rm bdd}(Y,\R)$ 
and $N=\dim H^k(X,\R)$. 
Put $|P|:=\sqrt{{\rm tr}({}^tPP)}$, which may depends on the choice 
of basis. 
Here, since $d$ may become infinity, 
we may have $|P|=\infty$. 

\begin{thm}
Let $(X^m,g)$ and $(Y^n,h)$ be Riemannian manifolds 
and $X$ be compact and oriented. 
For any $1\le k\le m$, there is a constant $C>0$ 
depending only on $k$, 
$(X,g)$, $(Y,h)$ and the basis of 
$H^k(X,\R)$, $H^k_{\rm bdd}(Y,\R)$ 
such that for any $f\in C^\infty(X,Y)$ we have 
\begin{align*}
\E_k(f)\ge C|P([f^*]^k_{\rm bdd})|.
\end{align*}
In particular, if $[f^*]^k_{\rm bdd}$ is a nonzero map, 
then the infimum of $\E_k|_{[f]}$ is positive. 
\label{thm lower bdd of k-energy}
\end{thm}
\begin{proof}
Take bounded 
closed $k$-forms $\beta_1,\ldots,\beta_d\in \Omega^k(Y)$ 
such that $\{ [\beta_l]\}_l$ is a basis 
of $H^k_{\rm bdd}(Y,\R)$. 

By the Hodge Theory, 
$H^k(X)$ is isomorphic to the space of 
harmonic $k$-forms as vector spaces. 
Therefore, for any basis of $H^k(X,\R)$, 
there is a corresponding basis 
$\alpha_1,\ldots,\alpha_N\in\Omega^k(X)$ 
of the space of harmonic $k$-forms. 
Let $G_{ij}:=\int_X\alpha_i\wedge *_g\alpha_j$, 
which is symmetric positive definite.

Define $P=(P_{ij})\in M_{N,d}(\R)$ 
by $[f^*]^k_{\rm bdd}([\beta_j])=\sum_iP_{ij}[\alpha_i]$. 
If we put 
\begin{align*}
\Phi:=\sum_{i,j} P_{ij}\beta_j\wedge (*_g\alpha_i),
\end{align*}
then every 
$\beta_j\wedge (*_g\alpha_i)$ is closed and satisfies the 
assumption of Lemma \ref{lem cal k}, 
since $X$ is compact and $\beta_j$ is bounded. 
Take the constant $C_{ij}>0$ as in Lemma \ref{lem cal k}. 
Here, $C_{ij}$ is depending only on 
$m,n,k$ and $\alpha_i,\beta_j$. 
Then for any $f\in C^\infty(X,Y)$, we have 
\begin{align*}
(1_X,f)^*\left\{ \beta_j\wedge (*_g\alpha_i)\right\}
&\le C_{ij}\sigma_k(f)\vol_g,\\
(1_X,f)^*\Phi&\le \sum_{i,j}C_{ij}|P_{ij}|\sigma_k(f)\vol_g\\
&\le \sqrt{\sum_{i,j}C_{ij}^2}|P|\sigma_k(f)\vol_g,
\end{align*}
hence 
\begin{align*}
\E_k(f)\ge \left(\sum_{i,j}C_{ij}^2\right)^{-1/2}
|P|^{-1}\int_X(1_X,f)^*\Phi.
\end{align*}
Moreover, we have 
\begin{align*}
\int_X(1_X,f)^*\Phi
&= \sum_{i,j} \int_X P_{ij}f^*\beta_j\wedge (*_g\alpha_i)\\
&=\sum_{i,j} \int_X P_{ij}\sum_kP_{kj}\alpha_k\wedge (*_g\alpha_i)\\
&=\sum_{i,j,k} P_{ij}P_{kj}G_{ki}.
\end{align*}
If we denote by $\lambda>0$ the minimum eigenvalue 
of $(G_{ij})_{i,j}$, then we have 
$\sum_{i,j,k} P_{ij}P_{kj}G_{ki}\ge \lambda|P|^2$. 
Hence we obtain 
\begin{align*}
\E_k(f)\ge \lambda\left(\sum_{i,j}C_{ij}^2\right)^{-1/2}
|P|.
\end{align*}
\end{proof}

\begin{rem}
\normalfont
Combining the above theorem 
with Proposition \ref{prop holder}, 
we also have the lower bound of $\E_p$ 
for any $p\ge k$. 
\end{rem}

\section{Energy of the identity maps}\label{sec id map}
In this section we consider when the identity map 
on compact oriented Riemannian manifold $X$ minimizes 
the energy. 
Here, we consider the family of energies. 
For Riemannian manifolds $(X^m,g)$, $(Y^n,h)$ and points 
$x\in X$, $y\in Y$, 
take a linear map $A\colon T_xX\to T_yY$. 
Fixing orthonormal basis of $T_xX$ and $T_yY$, 
we can regard $A$ as an $n\times m$-matirx.
Denote by $a_1,\ldots ,a_m\in\R_{\ge 0}$ 
the eigenvalues of ${}^tA\cdot A$, then put 
\begin{align*}
|A|_p:=\left( \sum_{i=1}^m a_i^{p/2}\right)^{1/p}
\end{align*}
for $p>0$.
Then $|A|_p$ is independent of 
the choice of the orthonormal basis of $T_xX$. 
For a smooth map $f\colon X\to Y$, 
let 
\begin{align*}
\sigma_{p,q}(f)|_x&:=|df_x|_p^q,\\
\E_{p,q}(f)&:=\int_X\sigma_{p,q}(f)\vol_g.
\end{align*}
Note that $\sigma_{2,p}=\sigma_p$ and $\E_{2,p}=\E_p$. 

From now on we consider $(Y,h)=(X,g)$ and a map 
$f\colon X\to X$. 
Let $1_X$ be the identity map of $X$. 
\begin{prop}
If $1_X$ minimizes $\E_{p,q}|_{[1_X]}$, 
then it also minimizes $\E_{p',q'}|_{[1_X]}$ 
for any $p'\ge p$ and $q'\ge q$. 
\end{prop}
\begin{proof}
First of all, for any smooth map $f$, we have 
\begin{align*}
|df_x|_p&\le m^{1/p-1/p'}| df_x|_{p'},\\
\E_{p,q}(f)&\le m^{q/p-q/p'}\vol_g(X)^{1-q/q'}
\left( \int_X|df|_{p'}^{q'}\vol_g\right)^{q/q'},
\end{align*}
by the H\"older's inequality, 
which gives $\E_{p',q'}(f)\ge C\E_{p,q}(f)^{q'/q}$ 
for some constant $C>0$. 
Moreover, we have the equality for $f=1_X$. 
Therefore, we can see 
\begin{align*}
\inf \E_{p',q'}|_{[1_X]}
\ge \inf C\E_{p,q}^{q'/q}|_{[1_X]}
= C\E_{p,q}(1_X)^{q'/q}= \E_{p',q'}(1_X)
\ge \inf \E_{p',q'}|_{[1_X]}.
\end{align*}
\end{proof}

\begin{prop}[{cf.\cite[Lemma 2.2]{hoisington2021}}]
Let $(X,g)$ be a compact oriented Riemannian manifold 
of dimension $m$. 
Then $1_X$ minimizes $\E_{1,m}$ in its homotopy class. 
\end{prop}
\begin{proof}
The proof is essentially given by 
\cite[Lemma 2.2]{hoisington2021}. 
For any map $f\colon X\to X$, we can see 
\begin{align*}
f^*\vol_g=\det(df)\vol_g
\le m^{-m}\sigma_{1,m}(f)\vol_g.
\end{align*}
Here, the second inequality follows from the inequality 
\begin{align*}
\frac{\sum_{i=1}^ma_i}{m}\ge \left( \prod_{i=1}^ma_i\right)^{1/m}
\end{align*}
for any $a_i\ge 0$. 
Therefore, we can see 
\begin{align*}
\E_{1,m}(f)\ge m^m\int_Xf^*\vol_g.
\end{align*}
Moreover, the equality holds if $f=1_X$. 
\end{proof}

Next we consider the analogy of the above proposition. 
We assume that $X$ has a nontrivial parallel $k$-form.

Denote by $g_0$ the standard metric on $\R^m$, 
which also induces the 
metric on $\Lambda^k(\R^m)^*$. 
Let $\varphi_0\in\Lambda^k(\R^m)^*$ and fix an orientation 
of $\R^m$. 
For a $k$-form $\varphi$ and a Riemannian metric $g$ on an 
oriented manifold $X$, 
we say that {\it 
$(g_0,\varphi_0)$ is a local model of $(g,\varphi)$} if 
for any $x\in X$ there is an orientation preserving 
isometry $I\colon \R^m\to T_xX$ 
such that $I^*(\varphi|_x)=\varphi_0$. 

Denote by $*_{g_0}\colon \Lambda^k(\R^m)^*\to \Lambda^{m-k}(\R^m)^*$ the Hodge star operator induced by the standard metric 
and let $\vol_{g_0}\in\Lambda^m(\R^m)^*$ be 
the volume form. 
First of all, we show the following proposition for the local model 
$(g_0,\varphi_0)$. 
\begin{prop}
Let $(g_0,\varphi_0)$ be as above. 
Assume that 
$\left| \iota_u\varphi_0\right|_{g_0}$ 
is independent of $u\in\R^m$ if $|u|_{g_0}=1$. 
We have 
\begin{align*}
A^*\varphi_0\wedge *_{g_0} \varphi_0\le \frac{|\varphi_0|_{g_0}^2}{m}|A|_k^k\vol_{g_0}
\end{align*}
for any $A\in M_m(\R)$. 
Moreover, if $A=\lambda T$ for $\lambda\in\R$, $T\in O(m)$ 
and $A^*\varphi_0=\lambda'\varphi_0$ for some 
$\lambda'\ge 0$, then we have the equality. 
\label{prop ptwise ineq}
\end{prop}
\begin{proof}
For any $A$, we can take oriented 
orthonormal basis 
$\{ e_1,\ldots,e_m\}$ and $e'_1,\ldots,e'_m$ 
of $(\R^m)^*$
such that $A^*e'_i=a_ie_i$ for some $a_i\in\R$. 
We put 
\begin{align*}
\varphi_0=\sum_{I\in\mathcal{I}_k^m}F_Ie_I=\sum_{I\in\mathcal{I}_k^m}F'_Ie'_I
\end{align*}
for some $F_I,F'_I\in\R$. 
Now, put $a_I:=a_{i_1}\cdots a_{i_k}$ 
for $I=(i_1,\ldots,i_k)\in\mathcal{I}_k^m$. 
The we have 
$A^*\varphi_0=\sum_IF'_Ia_Ie_I$ and 
\begin{align*}
A^*\varphi_0\wedge *_{g_0}\varphi_0
= g_0( A^*\varphi_0,\varphi_0)\vol_{g_0}
&=\sum_IF_IF'_Ia_I\vol_{g_0}\\
&\le \sum_I|F_IF'_I||a_I|\vol_{g_0}.
\end{align*}
If we put $\{ I\}:=\{ i_1,\ldots,i_k\}$, then 
\begin{align*}
|a_I|=\left( |a_{i_1}|^k\cdots |a_{i_k}|^k\right)^{1/k}
\le \frac{1}{k} \sum_{j\in \{ I\}}|a_j|^k,
\end{align*}
therefore, we obtain 
\begin{align*}
\sum_I|F_IF'_I||a_I|
&\le \frac{1}{k}\sum_I |F_IF'_I|\sum_{j\in \{ I\}}|a_j|^k\\
&= \frac{1}{k}\sum_{j=1}^m |a_j|^k\sum_{I\in\mathcal{I}_k^m,j\in\{ I\}} |F_IF'_I|.
\end{align*}
Denote by $\hat{g}_0\colon (\R^m)^*\to \R^m$ 
the isomorphism induced by the metric $g_0$. 
Put 
\begin{align*}
\varphi_1:=\sum_{I\in\mathcal{I}_k^m}|F_I|e_I,\quad
\varphi_2:=\sum_{I\in\mathcal{I}_k^m}|F'_I|e_I
\end{align*}
and define an orthogonal matrix  
$U\colon \R^m\to \R^m$ by $U\circ \hat{g}_0(e_j)=\hat{g}_0(e'_j)$. 
Now we can see 
\begin{align*}
\sum_{I\in\mathcal{I}_k^m,j\in\{ I\}} |F_IF'_I|
=g_0\left( \iota_{\hat{g}_0(e_j)}\varphi_1,\iota_{\hat{g}_0(e_j)}\varphi_2\right)
&\le \left| \iota_{\hat{g}_0(e_j)}\varphi_1\right|_{g_0}
\cdot \left| \iota_{\hat{g}_0(e_j)}\varphi_2\right|_{g_0}\\
&=\left| \iota_{\hat{g}_0(e_j)}\varphi_0\right|_{g_0}
\cdot \left| \iota_{\hat{g}_0(e_j)}(U^*\varphi_0)\right|_{g_0}
\end{align*}
and 
\begin{align*}
\left| \iota_{\hat{g}_0(e_j)}(U^*\varphi_0)\right|_{g_0}
=\left| U^*(\iota_{U\circ \hat{g}_0(e_j)}\varphi_0)\right|_{g_0}
=\left| \iota_{U\circ\hat{g}_0(e_j)}\varphi_0\right|_{g_0}.
\end{align*}
Then by the assumption, 
we can see that $C=\left| \iota_{\hat{g}_0(e_j)}\varphi_0\right|_{g_0}=\left| \iota_{U\circ\hat{g}_0(e_j)}\varphi_0\right|_{g_0}$ 
is independent of $j$, therefore, we have 
$\sum_{I\in\mathcal{I}_k^m,j\in\{ I\}} |F_IF'_I|\le C^2$ 
and 
\begin{align*}
A^*\varphi_0\wedge *_{g_0}\varphi_0
\le \frac{C^2}{k}\sum_{j=1}^m |a_j|^k\vol_{g_0}=\frac{C^2}{k}|A|_k^k\vol_{g_0}.
\end{align*}
In the above inequalities, we have the equality if $A=1_m$, 
then we can determine the constant $C$. 
Moreover, we can also check that 
the equality holds 
if $A=\lambda T$, where $\lambda\in\R$, $T\in O(m)$ 
and $A^*\varphi_0=\lambda'\varphi_0$ for some $\lambda'\ge 0$. 
\end{proof}

\begin{prop}
Let $(X^m,g)$ be a compact oriented Riemannian 
manifold and $\varphi\in \Omega^k(X)$ be 
a harmonic form. 
Assume that there is a local model $(g_0,\varphi_0)$ 
of $(g,\varphi)$ and $|\iota_u\varphi_0|_{g_0}$ 
is independent of $u\in\R^m$ if $|u|_{g_0}=1$. 
Denote by ${\rm pr}_i\colon X\times X\to X$ 
the projection to $i$-th component for 
$i=1,2$. 
Then $\Phi=m|\varphi_0|_{g_0}^{-2}{\rm pr}_2^*\varphi\wedge {\rm pr}_1^*(*_g\varphi)$ 
is an $\sigma_{k,k}$-calibration. 
Moreover, any isometry 
$f\colon X\to X$ with $f^*\varphi=\varphi$ 
is $(\sigma_{k,k},\Phi)$-calibrated. 
\label{prop id calib}
\end{prop}
\begin{proof}
$\Phi$ 
is an $\sigma_{k,k}$-calibration iff 
\begin{align*}
f^*\varphi\wedge *_g\varphi
\le \frac{|\varphi_0|_{g_0}^2}{m}|df|_k^k\vol_g.
\end{align*}
By putting $A=df_x$ and identifying 
$\R^m\cong T_xX$, this is equivalent to 
the inequality in Proposition \ref{prop ptwise ineq}. 
Moreover, the equality holds if 
$(df_x)^*\varphi|_x = \varphi|_x$ 
for all $x\in X$ and $df_x$ is isometry. 
\end{proof}

Next we have to consider when the assumption for 
$(g_0,\varphi_0)$ is satisfied. 
If $G\subset SO(m)$ is a closed subgroup, 
then the linear action of $SO(m)$ on $\R^m$ 
induces the action of $G$ on $\R^m$. 
Similarly, since $SO(m)$ acts on $\Lambda^k(\R^m)^*$ 
for all $k$, 
$G$ also acts on them. 
Here, $\R^m$ is {\it irreducible as a $G$-representation} 
if any subspace $W\subset \R^m$ which is closed 
under the $G$-action is equal to $\R^m$ or $\{ 0\}$. 
For $\varphi_0\in\Lambda^k(\R^m)^*$, 
denote by ${\rm Stab}(\varphi_0)\subset SO(m)$ 
the stabilizer of $\varphi_0$.

\begin{lem}
Let $G$ be a closed subgroup of $SO(m)$ 
and assume that $\R^m$ is irreducible 
as a $G$-representation. 
Moreover, assume that 
$G\subset {\rm Stab}(\varphi_0)$. 
Then $|\iota_u\varphi_0|_{g_0}$ 
is independent of $u\in\R^m$ if $|u|_{g_0}=1$. 
\label{lem holonomy irred}
\end{lem}
\begin{proof}

Define a linear map $\Psi\colon \R^m\to \Lambda^{k-1}(\R^m)^*$ 
by $\Psi(u):=\iota_u\varphi_0$, then we can see 
$\Psi$ is $G$-equivariant map since 
the $G$-action preserves $\varphi_0$. 
Since the $SO(m)$-action on $\Lambda^{k-1}(\R^m)^*$ preserves the inner product, we can see 
\begin{align*}
g_0( A\Psi(u),A\Psi(v) )
=g_0(  \Psi(u),\Psi(v) )
\end{align*}
for any $A\in G$ and $u,v\in\R^m$. 
Moreover, the left-hand-side is equal to 
$g_0( \Psi( Au),\Psi( Av) )$ since $\Psi$ 
is $G$-equivariant. 

Now, let $e_1,\ldots, e_m$ be the standard orthonormal 
basis of $\R^m$, and define the symmetric matrix 
$H=(H_{ij})_{i,j}$ by 
$H_{ij}:=g_0(\Psi(e_i),\Psi(e_j))$. 
Then by the above argument we have 
${}^tAHA=H$. 
Let $\lambda\in \R$ be any eigenvalue of $H$ 
and denote by $V(\lambda)\subset\R^m$ the 
eigenspace associate with $\lambda$. 
Then we can see that $V(\lambda)$ is closed under the $G$-action, 
hence we have $V(\lambda)=\R^m$ by the irreducibility, 
which implies 
\begin{align*}
|\Psi( Au)|_{g_0}^2=\lambda |u|_{g_0}^2
\end{align*}
for all $u\in\R^m$ and $A\in G$. 
\end{proof}

Let $(X^m,g)$ be an oriented Riemannian manifold 
and denote by ${\rm Hol}_g\subset SO(m)$ the 
holonomy group. 
We consider 
$(X,g,\varphi,G,g_0,\varphi_0)$, 
where $\varphi\in\Omega^k(X)$ is closed, 
$(g_0,\varphi_0)$ is a local model of $(g,\varphi)$ 
and $G$ is a closed subgroup of $SO(m)$ 
such that ${\rm Hol}_g\subset G\subset {\rm Stab}(\varphi_0)$. 
The followings are examples. 

\begin{table}[h]
  \caption{Examples of $(X,g,\varphi,G,g_0,\varphi_0)$}
  \label{table_holonomy}
  \centering
  \begin{tabular}{lccl}
\hline
 $(X,g,\varphi)$ & $m$ & $G$ & $k$ \\
\hline \hline
 K\"ahler manifold & $2q$ & $U(q)$ & $2$\\
\hline
 quaternionic K\"ahler manifold & $4q\ge 8$ & $Sp(q)\cdot Sp(1)$ & $4$\\
\hline
$G_2$ manifold & $7$ & $G_2$ & $3$\\
\hline
$Spin(7)$ manifold & $8$ & $Spin(7)$ & $4$\\
\hline
  \end{tabular}\label{table holonomy}
\end{table}
We can apply Proposition \ref{prop id calib} and 
Lemma 
\ref{lem holonomy irred} 
to the above cases and obtain the following result. 
\begin{thm}
Let $(X,g,\varphi)$ be an oriented compact Riemannian 
manifold whose geometric structure is one of Table 
\ref{table holonomy}
and let $\Phi$ be as in Proposition \ref{prop id calib}. 
Then the identity map $1_X$ is a 
$(\sigma_{k,k},\Phi)$-calibrated map. 
In particular, $1_X$ minimizes 
$\E_{k,k}$ in its homotopy class. \label{thm id calib}
\end{thm}

\section{Intersection of Smooth maps}\label{sec intersection}
In \cite{CrokeFathi1990}, 
Croke and Fathi introduced the homotopy invariant 
of a smooth map $f\colon X\to Y$ which gives the lower 
bound to the $2$-energy $\E_2$. 
In this section, we compare our invariant with the invariant 
in \cite{CrokeFathi1990}.

First of all, we review the intersection of 
smooth map introduced in \cite{CrokeFathi1990}. 
Let $(X,g)$ and $(Y,h)$ be Riemannian manifolds 
and suppose $X$ is compact. 
Here, we do not assume $X$ is oriented, 
and we use the volume measure $\mu_g$ of $g$ 
instead of the volume form. 

Croke and Fathi defined the following quantity 
\begin{align*}
i_f(g,h)=\lim_{t\to \infty}\frac{1}{t}\int_{S_g(X)}\phi_t(v)d
{\rm Liou}_g(v)
\end{align*}
for a smooth map $f\colon X\to Y$ and called it 
{\it the intersection of $f$}. 
Here, ${\rm Liou}_g$ is the Liouville measure on 
the unit tangent bundle $S_g(X)$ and 
$\phi_t^f(v)=\phi_t(v)$ is the minimum length of 
all paths in $Y$ homotopic with the fixed endpoints 
to 
\begin{align*}
s\mapsto f(\gamma_v(s)),\quad 0\le s\le t,
\end{align*}
where $\gamma_v$ is the geodesic 
from $p\in X$ with $\gamma_v'(0)=v\in S_g(X)$.

\begin{thm}[\cite{CrokeFathi1990}]
For a smooth map $f\colon X\to Y$, 
the intersection $i_f(g,h)$ is homotopy invariant, 
that is, $i_f(g,h)=i_{f'}(g,h)$ if $[f']=[f]$. 
Moreover, for any $f$, we have 
\begin{align*}
\int_X \sigma_2(f)d\mu_g
\ge\frac{m}{V(S^{m-1})^2\mu_g(X)} i_f(g,h)^2,
\end{align*}
where $V(S^{m-1})$ is the volume of 
the unit sphere $S^{m-1}$ in $\R^m$.
\label{thm croke fathi}
\end{thm}
First of all, we introduce the variant of $i_f(g,h)$ and 
improve the above theorem. 
We put 
\begin{align*}
j_f(g,h):=\lim_{t\to \infty}\frac{1}{t^2}\int_{S_g(X)}\phi_t(v)^2d
{\rm Liou}_g(v).
\end{align*}

\begin{thm}
For a smooth map $f\colon X\to Y$, 
$j_f(g,h)$ is homotopy invariant. 
Moreover, for any $f$, we have 
\begin{align*}
\int_X \sigma_2(f)d\mu_g
\ge\frac{m}{V(S^{m-1})} j_f(g,h),
\end{align*}
where the equality holds iff 
the image of the geodesic in $X$ by $f$ 
minimizes the length in its homotopy class with the fixed endpoints. 
\label{thm croke fathi2}
\end{thm}
\begin{proof}
The proof is parallel to that of Theorem 
\ref{thm croke fathi}. 
The homotopy invariance of 
$j_f(g,h)$ is same as the case of $i_f(g,h)$. 
See the proof of \cite[Lemma 1.3]{CrokeFathi1990}. 

Next we show the inequality. 
Here we can see 
\begin{align*}
\int_X\sigma_2(f)d\mu_g
&=\frac{m}{V(S^{m-1})}\int_{S_g(X)} | df(v) |_h^2d{\rm Liou}_g(v).
\end{align*}
For $s\ge 0$, let $g_s\colon S_g(X)\to S_g(X)$ be the 
geodesic flow. 
Since $g_s$ preserves the Liouville measure, 
we can see 
\begin{align*}
\int_{S_g(X)}|df(v)|_h^2d{\rm Liou}_g(v)
&=\frac{1}{t}\int_{S_g(X)}\left(\int_0^t|df(g_sv)|_h^2
ds\right)d{\rm Liou}_g(v)\\
&=\frac{1}{t}\int_{S_g(X)}\mathcal{E}_2(f\circ\gamma_v|_{[0,t]})d{\rm Liou}_g(v),
\end{align*}
where $\mathcal{E}_2$ is the $2$-energy of the curves in 
$(Y,h)$. If $L(c)$ is the length of $c$, 
then we have 
\begin{align*}
\mathcal{E}_2(c)
=\int_a^b|c'(s)|_h^2ds
&\ge \frac{1}{b-a}\left(\int_a^b|c'(s)|_hds\right)^2\\
&= \frac{1}{b-a}\min_cL(c)^2,
\end{align*}
therefore 
\begin{align*}
\int_{S_g(X)}|df(v)|_h^2d{\rm Liou}_g(v)
&\ge\frac{1}{t^2}\int_{S_g(X)}\phi_t(v)^2d{\rm Liou}_g(v)
\end{align*}
for any $t>0$. 
Consequently, we have the second assertion by considering $t\to \infty$. 
Finally, we consider the condition when 
\begin{align*}
\int_{S_g(X)}|df(v)|_h^2d{\rm Liou}_g(v)
&= \lim_{t\to\infty}\frac{1}{t^2}\int_{S_g(X)}\phi_t(v)^2d{\rm Liou}_g(v)
\end{align*}
holds. 
To consider it, we show 
\begin{align}
\lim_{t\to\infty}\frac{1}{t^2}\int_{S_g(X)}\phi_t(v)^2d{\rm Liou}_g(v)
=\inf_{t>0}\frac{1}{t^2}\int_{S_g(X)}\phi_t(v)^2d{\rm Liou}_g(v).
\label{eq lim inf}
\end{align}
By \cite[Lemma 1.2]{CrokeFathi1990}, 
we have 
\begin{align*}
\phi_{t+t'}(v)\le \phi_{t'}(g_tv)+\phi_t(v)
\end{align*}
for any $t,t'\ge 0$. Then by combining the Cauchy-Schwarz 
inequality, we have 
\begin{align*}
\int_{S_g(X)}\phi_{t+t'}(v)^2d{\rm Liou}_g(v)
&\le
\int_{S_g(X)}\phi_{t'}(g_tv)^2d{\rm Liou}_g(v)\\
&\quad\quad
+2\sqrt{
\int_{S_g(X)}\phi_{t'}(g_tv)^2d{\rm Liou}_g(v)
\int_{S_g(X)}\phi_t(v)^2d{\rm Liou}_g(v)
}\\
&\quad\quad
+\int_{S_g(X)}\phi_t(v)^2d{\rm Liou}_g(v).
\end{align*}
Since the Liouville measure is invariant 
under the geodesic flow, 
we can see 
$\int_{S_g(X)}\phi_{t'}(g_tv)^2d{\rm Liou}_g(v)
=\int_{S_g(X)}\phi_{t'}(v)^2d{\rm Liou}_g(v)$, 
hence 
\begin{align*}
\int_{S_g(X)}\phi_{t+t'}(v)^2d{\rm Liou}_g(v)
&\le
\left( 
\sqrt{ \int_{S_g(X)}\phi_{t'}(v)^2d{\rm Liou}_g(v)}
+
\sqrt{ \int_{S_g(X)}\phi_t(v)^2d{\rm Liou}_g(v)}
\right)^2.
\end{align*}

If we put 
\begin{align*}
P_t:=\sqrt{ \int_{S_g(X)}\phi_t(v)^2d{\rm Liou}_g(v)},
\end{align*}
then we have $P_{t+t'}\le P_t+P_{t'}$, 
hence 
\begin{align*}
\inf_{t>0}\frac{P(t)}{t}
=\lim_{t\to\infty}\frac{P(t)}{t}.
\end{align*}
Thus we obtain \eqref{eq lim inf}. 

Now, suppose 
\begin{align*}
\int_X \sigma_2(f)d\mu_g
= \frac{m}{V(S^{m-1})} j_f(g,h).
\end{align*}
By the above argument, 
we can see that $f\circ \gamma_v|_{[0,t]}$ is 
geodesic for any $v\in S_g(X)$ and $t>0$, 
and $L(f\circ \gamma_v|_{[0,t]})$ 
gives the minimum of 
\begin{align*}
\left\{ L(c)|\, c\mbox{ is homotopic with the fixed endpoints to }f\circ \gamma_v|_{[0,t]}\right\}.
\end{align*}
\end{proof}

\begin{rem}
\normalfont
By the Cauchy-Schwarz inequality, we have 
\begin{align*}
j_f(g,h)\ge \frac{i_f(g,h)^2}{\mu_g(X)V(S^{m-1})},
\end{align*}
therefore, the inequality in Theorem \ref{thm croke fathi2} 
implies the inequality in 
Theorem \ref{thm croke fathi}. 
\end{rem}

Next we compute $j_f(g,h)$ in the case of flat tori, 
and compare with the lower bound obtained by Proposition \ref{prop lower torus}. 
Let $(\T^m,g)$ and $(\T^n,h)$ be as in 
Subsection \ref{subsec tori}, 
and take a coordinate $x$ on $\T^m$ 
and $y$ on $\T^n$ as in 
Subsection \ref{subsec tori}. 

\begin{prop}
Let $f\colon \T^m\to \T^n$ be a smooth map 
such that $f^*([dy^j])=\sum_iP_i^j[dx^i]$ 
for $P=(P_i^j)\in M_{m,n}(\Z)$. 
If we define $\Phi$ by 
\eqref{eq phi on torus} in Subsection \ref{subsec tori}, 
then we have 
\begin{align*}
j_f(g,h)=\frac{V(S^{m-1})}{m}\int_{\T^m}(1_{\T^m},f)^*\Phi.
\end{align*}
\end{prop}
\begin{proof}
First of all, we can see that 
$f$ is homotopic to the map 
given by $x\mapsto Px$ for $x\in\T^m$, 
hence it suffices to show the equality 
by putting $f(x)={}^tPx$.  

Since the image of the geodesic by $f$ 
minimizes the length in its homotopy class with the fixed endpoints, then by Theorem 
\ref{thm croke fathi2}, 
we have 
$\E_2(f)=\frac{m}{V(S^{m-1})}j_f(g,h)$. 
we can compute $\E_2(f)$ directly as 
\begin{align*}
\E_2(f)=\int_{\T^m}|df|^2\vol_g
=\sum_{i,j,k,l}h_{ij}P_k^iP_l^jg^{kl}\vol_g(\T^m)
=\| P\|^2\vol_g(\T^m).
\end{align*}
Moreover, 
by the computation in Subsection \ref{subsec tori}, 
we have shown that 
\begin{align*}
\int_{\T^m}(1_{\T^m},f)^*\Phi=\| P\|^2\vol_g(\T^m).
\end{align*}
Therefore,  
\begin{align*}
\int_{\T^m}(1_{\T^m},f)^*\Phi
=\| P\|^2\vol_g(\T^m)
=\frac{m}{V(S^{m-1})}j_f(g,h).
\end{align*}
\end{proof}

\bibliographystyle{plain}

\begin{comment}
\begin{align*}
\end{align*}
\[ A=
\left (
\begin{array}{ccc}
 &  &  \\
 &  & \\
 &  & 
\end{array}
\right ).
\]

\begin{itemize}
\setlength{\parskip}{0cm}
\setlength{\itemsep}{0cm}
 \item[(1)] 
 \item[(2)] 
\end{itemize}

\begin{thm}
\end{thm}
\begin{prop}
\end{prop}
\begin{proof}
\end{proof}

\begin{definition}
\normalfont
\end{definition}
\begin{rem}
\normalfont
\end{rem}
\begin{lem}
\end{lem}

\end{comment}
\end{document}